\documentclass[11pt]{article}

\usepackage{latexsym,epsfig} 
\usepackage{amsmath}
\usepackage{amsthm, amssymb}
\usepackage{mathtools}
\usepackage{graphicx}   
\usepackage{url}
\usepackage[utf8]{inputenc} 
\usepackage[affil-it]{authblk}
\usepackage{mathrsfs} 
\usepackage{tikz}
\usepackage{ytableau} 
\usepackage{caption}

\usetikzlibrary{arrows,automata}

\title{The Limit Shape of a Stochastic Bulgarian Solitaire}

\author[1]{Kimmo Eriksson}
\author[1]{Markus Jonsson}
\author[2]{Jonas Sjöstrand}
\affil[1]{Mälardalen University, School of Education, Culture and Communication, Box 883, SE-72123 Västerås, Sweden}
\affil[2]{Royal Institute of Technology, Department of Mathematics, SE-10044 Stockholm, Sweden}

\newtheorem{theorem}{Theorem}
\newtheorem{proposition}{Proposition}
\newtheorem{lemma}{Lemma}

\newtheorem{corollary}{Corollary}
\newtheorem{observation}{Observation}

\theoremstyle{definition}
\newtheorem{definition}{Definition}
\newtheorem{example}{Example}

\newcommand{\NN}{\mathbb{N}}
\newcommand{\ZZ}{\mathbb{Z}}
\newcommand{\RR}{\mathbb{R}}
\newcommand{\UD}{\textit{UD}}
\newcommand{\eps}{\varepsilon}
\newcommand{\Prob}{{\rm Prob}}
\newcommand{\itemb}{\item[$\bullet$]}

\begin{document}

\title{Markov chains on graded posets: Compatibility of up-directed and down-directed transition probabilities}
\maketitle

\begin{abstract}
We consider two types of discrete-time Markov chains where the
state space is a graded poset and the transitions
are taken along the covering relations in the poset. The first type
of Markov chain goes only in one direction, either up or down in the poset
(an \emph{up chain} or \emph{down chain}). The second type
toggles between two adjacent rank levels (an \emph{up-and-down chain}).

We introduce two compatibility concepts between the up-directed
transition probabilities (an \emph{up rule}) and the down-directed
(a \emph{down rule}), and we relate these to compatibility between
up-and-down chains. This framework is used to prove a conjecture
about a limit shape for a process on Young's lattice.

Finally, we settle the questions whether the reverse of an
up chain is a down chain for some down rule and whether there
exists an up or down chain at all if the rank function is not
bounded.
\end{abstract}

\section{Introduction}
\label{sec:intro}
A Hasse walk is a walk along the covering relations in a poset
\cite{stanley1996enumerative1,stanley1988differential}. 
In \cite{Eriksson2012575}, Eriksson and Sjöstrand discussed how several
famous models of stochastic processes can be regarded as random Hasse walks
on Young's lattice, either walks that go steadily upwards (e.g.\ Simon's
model of urban growth) or that alternately go up and down (e.g.\ the 
Moran model in population genetics). The aim of the present paper is
to develop a general framework of such unidirected and alternatingly
directed random Hasse walks on graded posets.

Let $I\subseteq\ZZ$ be a (possibly infinite) interval of the integers.
An \emph{$I$-graded poset} $\Omega$ is a countable (or finite)
poset together with a surjective
map $\rho\colon\Omega\to I$, called
the \emph{rank function}, such that
\begin{itemize}
\itemb
$u<v$ implies $\rho(u)<\rho(v)$, and
\itemb
$u\lessdot v$ implies $\rho(v)=\rho(u)+1$,
\end{itemize}
where $u\lessdot v$ means that $v$ \emph{covers} $u$, that is $u<v$ but there is no $w$ with $u<w<v$.
We can partition $\Omega=\bigcup_{i\in I}\Omega_i$ into
its \emph{level sets} $\Omega_i=\rho^{-1}(i)$.

We will now describe two types of stochastic processes on $\Omega$.

\subsection{Up chains and down chains}
\label{sec:up-processes}
An assignment of a probability $T(u\rightarrow v)$ to any pair
$(u,v)\in\Omega\times\Omega$ is an
\emph{up rule} if
\begin{itemize}
\itemb
$T(u\rightarrow v)>0\ \Leftrightarrow\ u\lessdot v$, and
\itemb
$\sum_{v\in\Omega}T(u\rightarrow v)=1$ for any $u$ with
non-maximal rank.
\end{itemize}
Analogously, it is a \emph{down rule} if
\begin{itemize}
\itemb
$T(u\rightarrow v)>0\ \Leftrightarrow\ u\gtrdot v$, and
\itemb
$\sum_{v\in\Omega}T(u\rightarrow v)=1$ for any $u$ with
non-minimal rank.
\end{itemize}

\begin{example}
Figure \ref{fig:exup} shows an example of a $[0,2]$-graded poset
with two sets of probabilities forming an up rule (left) and a down rule (right).
\end{example}
\captionsetup{justification=raggedright,singlelinecheck=false}
\begin{figure}[h]
\begin{tikzpicture}
\draw (0,0) -- (-2,2);
\draw (0,0) -- (2,2);
\draw (-1,1) -- (0,2);
\draw (1,1) -- (0,2);

\draw[->] (0,0) -- (-0.5,0.5);
\draw[->] (0,0) -- (0.5,0.5);
\draw[->] (-1,1) -- (-0.5,1.5);
\draw[->] (-1,1) -- (-1.5,1.5);
\draw[->] (1,1) -- (0.5,1.5);
\draw[->] (1,1) -- (1.5,1.5);

\draw[fill] (0,0) circle (.06);
\draw[fill] (-1,1) circle (.06);
\draw[fill] (-2,2) circle (.06);
\draw[fill] (0,2) circle (.06);
\draw[fill] (1,1) circle (.06);
\draw[fill] (2,2) circle (.06);

\node[left] at (-0.5,0.5) {$\frac{7}{10}$};
\node[right] at (0.5,0.5) {$\frac{3}{10}$};
\node[left] at (-1.5,1.45) {$\frac{3}{4}$};
\node[left] at (-0.5,1.55) {$\frac{1}{4}$};
\node[right] at (0.5,1.55) {$\frac{2}{5}$};
\node[right] at (1.5,1.45) {$\frac{3}{5}$};


\draw (6,0) -- (4,2);
\draw (6,0) -- (8,2);
\draw (5,1) -- (6,2);
\draw (7,1) -- (6,2);

\draw[->] (5,1) -- (5.5,0.5);
\draw[->] (7,1) -- (6.5,0.5);
\draw[->] (6,2) -- (5.5,1.5);
\draw[->] (4,2) -- (4.5,1.5);
\draw[->] (6,2) -- (6.5,1.5);
\draw[->] (8,2) -- (7.5,1.5);

\draw[fill] (6,0) circle (.06);
\draw[fill] (5,1) circle (.06);
\draw[fill] (4,2) circle (.06);
\draw[fill] (6,2) circle (.06);
\draw[fill] (7,1) circle (.06);
\draw[fill] (8,2) circle (.06);

\node[left] at (5.5,0.5) {$1$};
\node[right] at (6.5,0.5) {$1$};
\node[left] at (4.5,1.45) {$1$};
\node[left] at (5.5,1.55) {$\frac{3}{4}$};
\node[right] at (6.5,1.55) {$\frac{1}{4}$};
\node[right] at (7.5,1.45) {$1$};


\end{tikzpicture}
\caption{An up rule (left) and  a down rule (right) on some $\Omega$.}
\label{fig:exup}
\end{figure}
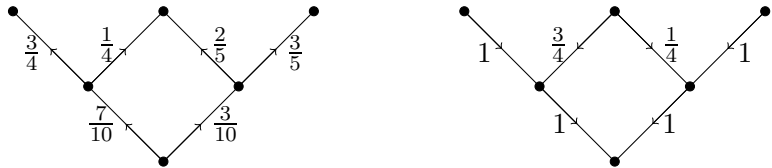

Up rules and down rules define classes of Markov chains on $\Omega$.
We will consider such Markov chains $(X_i)_{i\in J}$ on $\Omega$ for
any (possibly infinite) subinterval $J\subseteq I$.

\begin{definition}
Let $U$ be an up rule and let $D$ be a down rule on an $I$-graded
poset $\Omega$. A Markov chain $(X_i\in\Omega_i)_{i\in J}$ is a
\emph{$U$-chain (with time interval $J$)} if
\[
\Prob(X_i=u\ \text{and}\ X_{i+1}=v)
=\Prob(X_i=u)U(u\rightarrow v)
\]
for any non-maximal $i\in J$ and any $u,v\in\Omega$.
Analogously, it is a \emph{$D$-chain} if
\[
\Prob(X_i=u\ \text{and}\ X_{i-1}=v)
=\Prob(X_i=u)D(u\rightarrow v)
\]
for any non-minimal $i\in J$ and any $u,v\in\Omega$.

A $U$- or $D$-chain is \emph{maximal} if its time interval is
$I$, and it is \emph{positive} if $\Prob(X_i=u)>0$ for any $i\in J$
and any $u\in\Omega_i$.
\end{definition}

In this paper we shall examine when an up rule and a down rule are compatible
with each other. We shall distinguish between a weaker and a stronger notion
of compatibility.

\begin{definition}
An up rule $U$ and a down rule $D$ on $\Omega$ are \emph{compatible}
if there is a maximal $U$-chain $(X_i\in\Omega_i)_{i\in I}$
and a maximal $D$-chain $(Y_i\in\Omega_i)_{i\in I}$ such that,
for any $i\in I$, the random variables $X_i$ and $Y_i$ are
equally distributed.

$U$ and $D$ are \emph{strongly compatible} if
there is a Markov chain that
is both a maximal $U$-chain and a maximal $D$-chain.
\end{definition}

\begin{example}
Let $U$ and $D$ be the up and down rules depicted in Figure~\ref{fig:excompat}.
Clearly, $U$ and $D$ are compatible --- just assign the probability
$1/2$ to each element in the poset. But they are not strongly
compatible since the probability of going diagonally is $3/4$ for
the up rule and $1/4$ for the down rule.
\end{example}

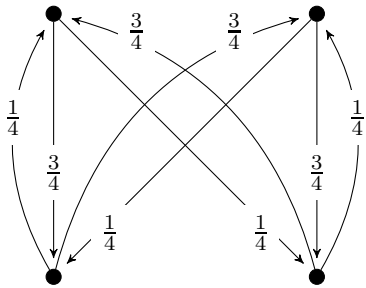
\begin{figure}[h]
\begin{tikzpicture}[->,>=stealth',shorten >=4pt,auto,node distance=3.5cm,thin]
\tikzstyle{every state}=[fill=none,draw=none,text=black,radius=1cm]

\node[state,inner sep=2pt,minimum size=2pt] (u1) {};
\node[state,inner sep=2pt,minimum size=2pt] (v1) [above of=u1] {};
\node[state,inner sep=2pt,minimum size=2pt] (u2) [right of=u1] {};
\node[state,inner sep=2pt,minimum size=2pt] (v2) [right of=v1] {};

\draw[fill] (v1) circle (0.1);
\draw[fill] (v2) circle (0.1);
\draw[fill] (u1) circle (0.1);
\draw[fill] (u2) circle (0.1);

\path (u1) edge [bend left] node [fill=white,pos=0.5,above] {$\frac{1}{4}$} (v1)
           edge [bend left] node [fill=white,pos=0.8,above=-6pt] {$\frac{3}{4}$} (v2)
      (u2) edge [bend right] node [fill=white,pos=0.8,above=-6pt] {$\frac{3}{4}$} (v1)
           edge [bend right] node [fill=white,pos=0.5,above] {$\frac{1}{4}$} (v2);

\path (v1) edge [left] node [fill=white,pos=0.5,below] {$\frac{3}{4}$} (u1)
      (v1) edge [left] node [fill=white,pos=0.8,below=-6pt] {$\frac{1}{4}$} (u2)
      (v2) edge [right] node [fill=white,pos=0.8,below=-6pt] {$\frac{1}{4}$} (u1)
      (v2) edge [left] node [fill=white,pos=0.5,below] {$\frac{3}{4}$} (u2);

\end{tikzpicture}
\caption{The up rule (bent arrows) and the down rule (straight arrows) are compatible but not strongly compatible.}
\label{fig:excompat}
\end{figure}

In order to study probability distributions on the poset $\Omega$ we introduce some
notation.

Let $\ell_1(\Omega)$ denote the Banach space of real-valued
functions $f$ on $\Omega$ such that the norm
$\lVert f \rVert=\sum_{u\in\Omega}\lvert f(u)\rvert$ is finite.
Let
\[
\ell_1(\Omega_i)=\{f\in\ell_1(\Omega)\,:\,
f^{-1}(\RR\setminus\{0\})\subseteq\Omega_i\}
\]
denote the subspace consisting of functions with support on the level set $\Omega_i$
and let
\[
M(\Omega_i)=\{\pi\in\ell_1(\Omega_i)\,:\,\pi(u)\ge0\ \text{for any}\ u\in\Omega_i\ \text{and}\ \sum_{u\in\Omega_i}\pi(u)=1\}
\]
denote the set of probability distributions on $\Omega_i$.

An up or down rule $T$ induces a linear
operator (which, by abuse of notation, also is called $T$)
on $\ell_1(\Omega)$ defined by
\[
(T\pi)(v)=\sum_{u\in\Omega}T(u\rightarrow v)\pi(u).
\]
Stepwise application of this operator defines sequences of probability distributions
with support on one level set at a time, as follows.

\begin{definition}
Let $U$ be an up rule and let $D$ be a down rule on an $I$-graded
poset $\Omega$.
A sequence $(\pi_i\in M(\Omega_i))_{i\in J}$ for some
(possibly infinite) interval $J\subseteq I$
is a \emph{$U$-sequence} if
$U\pi_i=\pi_{i+1}$ for any non-maximal $i\in J$, and it is
a \emph{$D$-sequence} if
$D\pi_i=\pi_{i-1}$ for any non-minimal $i\in J$. 

A $U$- or $D$-sequence $(\pi_i)_{i\in J}$
is \emph{positive} if $\pi_i(u)>0$ for any $i\in J$ and
any $u\in\Omega_i$.
\end{definition}
Clearly,
there is a one-to-one correspondence between $U$-sequences and
$U$-chains and between $D$-sequences and $D$-chains.

\begin{observation}
$U$ and $D$ are compatible if and only if there exists a $U$-sequence that is
also a $D$-sequence.
\end{observation}

\subsection{Up-and-down processes}
Next we turn to alternatingly directed Markov chains on $\Omega$.

Given an up rule $U$ and a down rule $D$, we define
a \emph{\UD-chain} as
a Markov chain with state space $\Omega$ and with
transitions induced by $U$ and $D$ alternately.
\begin{definition}
Let $U$ be an up rule and let $D$ be a down rule on an $I$-graded
poset $\Omega$. A \emph{\UD-chain} is
a Markov chain $(X^{(0)}, X^{(1)}, \dotsc)$ on $\Omega$ such that
\[
\Prob(X^{(t)}=u\ \text{and}\ X^{(t+1)}=v)
=\begin{cases}
\Prob(X^{(t)}=u)U(u\rightarrow v) & \text{if $t$ is even} \\
\Prob(X^{(t)}=u)D(u\rightarrow v) & \text{if $t$ is odd}.
\end{cases}
\]
\end{definition}

Figure~\ref{fig:exupdown} depicts the level sets $\Omega_1=\{u_1,u_2\}$ and $\Omega_2=\{v_1,v_2,v_3\}$ of the poset $\Omega$ in Figure~\ref{fig:exup} with the same up rule $U$ and down rule $D$. The up rule $U$ and the down rule $D$ define a \UD-chain with state space $\Omega_1 \cup \Omega_2$.

\begin{figure}[h]
\begin{tikzpicture}[->,>=stealth',shorten >=1pt,auto,node distance=3cm,semithick]
\tikzstyle{every state}=[fill=none,draw=black,text=black,radius=3cm]


\node[state,inner sep=2pt,minimum size=2pt] (v1) {$v_1$};
\node[state,inner sep=2pt,minimum size=2pt] (u1) [below right of=v1] {
$u_1$
};
\node[state,inner sep=2pt,minimum size=2pt] (v2) [above right of=u1] {$v_2$};
\node[state,inner sep=2pt,minimum size=2pt] (u2) [below right of=v2] {$u_2$};
\node[state,inner sep=2pt,minimum size=2pt] (v3) [above right of=u2] {$v_3$};

\path (u1) edge [bend right] node [pos=0.4,above] {$\frac{3}{4}$} (v1)
           edge [bend left] node [pos=0.4,above] {$\frac{1}{4}$} (v2)
      (u2) edge [bend right] node [pos=0.4,above] {$\frac{2}{5}$} (v2)
           edge [bend left] node [pos=0.4,above] {$\frac{3}{5}$} (v3);

\path (v1) edge [bend right] node [pos=0.4,below] {$1$} (u1)
      (v2) edge [bend left] node [pos=0.4,below] {$\frac{3}{4}$} (u1)
      (v2) edge [bend right] node [pos=0.4,below] {$\frac{1}{4}$} (u2)
      (v3) edge [bend left] node [pos=0.4,below] {$1$} (u2);
\end{tikzpicture}
\caption{A \UD-chain.}
\label{fig:exupdown}
\end{figure}

\subsection{Results}
\label{sec:compat}
Our first result relates the \UD-process to the compatibility of $U$- and $D$-chains. We state this as two theorems, one about strong compatibility and one about compatibility. These theorems follow almost immediately from the definitions, but nevertheless they can be powerful tools. In Section \ref{sec:limithapes} we use Theorem~\ref{th:seqandstatdistr} to prove a conjecture of Eriksson and Sjöstrand \cite{Eriksson2012575} about the limit shape of a Markov process on Young's lattice.

\begin{theorem}\label{th:chainsandtoggling}
Let $\Omega$ be an $I$-graded poset with an up rule $U$ and a down rule
$D$, and let $(X_i\in\Omega_i)_{i\in I}$ be a Markov chain. The following are equivalent.
\begin{itemize}
\item[(a)]
$(X_i)_{i\in I}$ is both a $U$-chain and a $D$-chain (and hence
$U$ and $D$ are strongly compatible).
\item[(b)]
For any adjacent levels $i,i+1\in I$ it holds that
$X_i,X_{i+1},X_i,X_{i+1},\dotsc$ is a \UD-chain.
\end{itemize}
\end{theorem}
\begin{proof}
This follows directly from the definitions.
\end{proof}
In other words, Theorem~\ref{th:chainsandtoggling} says that if and only if random
variables can be defined on each level set such that they correspond both to a
Markov chain following the up rule and a Markov chain following the down rule, they
also correspond to all Markov chains following the up and down rule alternatingly.

\begin{theorem}\label{th:seqandstatdistr}
Let $\Omega$ be an $I$-graded poset with an up rule $U$ and a down rule
$D$, and let $(\pi_i\in M(\Omega_i))_{i\in I}$ be a sequence of
probability distributions. The following are equivalent.
\begin{itemize}
\item[(a)]
$(\pi_i)_{i\in I}$ is both a $U$-sequence and a $D$-sequence (and hence
$U$ and $D$ are compatible).
\item[(b)]
For any adjacent levels $i,i+1\in I$ it holds that
$(\pi_i+\pi_{i+1})/2$ is a stationary distribution of
the \UD-process.
\end{itemize}
\end{theorem}
\begin{proof}
Note that $(\pi_i+\pi_{i+1})/2$ is a stationary distribution
of the \UD-process if and only if $U\pi_i=\pi_{i+1}$ and $D\pi_{i+1}=\pi_i$.
Now the theorem follows directly from the definitions.
\end{proof}
In words, Theorem~\ref{th:seqandstatdistr} says that if and only if a given sequence of
probability distributions is both a $U$-sequence and a $D$-sequence, the average
(in $\ell_1(\Omega)$) of two adjacent such distributions is a stationary distribution of the \UD-process.

The second part of this paper will answer the following questions that arise naturally from the notion of compatibility.

\begin{itemize}
\item[Q1.]
Given an up rule $U$ and a $U$-chain $(X_i)$,
is there a down-rule $D$ such that $(X_i)$ is a $D$-chain?
\item[Q2.]
Given an up or down rule $T$, is there a maximal $T$-chain?
\end{itemize}

The answer to Q1 is that every \emph{positive} $U$-chain
is a $D$-chain for some $D$.
This is Theorem~\ref{th:existencedownrule} in Section \ref{sec:result2} and we give a constructive proof using Bayesian updating of probability distributions between adjacent levels.

The answer to Q2 is more complex and
is presented as three theorems in Section~\ref{sec:result3}.
If all level sets $\Omega_i$ are finite, it turns out that there always exists a
$T$-chain, but not necesserily a positive one.
The proofs rely on topological arguments and require the axiom of choice.

\subsection{Applications}
We will apply the results to processes on two particular posets:
(i) Young's lattice; and (ii) the $d$-dimensional nonnegative
integer lattice.

In \cite{Eriksson2012575}, Eriksson and Sjöstrand study stochastic
processes on Young diagrams of integer partitions, in particular
their limit shapes. 
In the framework of this paper, these processes are up processes and
up-and-down processes, respectively, where the underlying poset is Young's
lattice. A motivation for this work is a limit shape conjecture in
\cite{Eriksson2012575} for a certain up-and-down process (described
in Section \ref{sec:rowderow} below). Using the results in the current
paper, in particular Theorem~\ref{th:seqandstatdistr}, we will prove this conjecture.
This is done in Corollary~\ref{cor:iff2} in Section \ref{sec:limithapes}.

Applications of our theory to the $d$-dimensional nonnegative integer
lattice are presented in Section \ref{sec:lattice}.

\subsection{An application of Theorem~\ref{th:seqandstatdistr} to a process on Young's lattice}
\label{sec:rowderow}
One of the processes studied in \cite{Eriksson2012575}, called $\textsc{derow-row}(\mu)$, is a \UD-chain on Young's lattice. 
We will first introduce some notation for integer partitions and Young diagrams and then describe the up rule $\textsc{row}(\mu)$ and the down rule \textsc{derow} used in this process.

\subsubsection{Notation}
An introduction to the theory of integer partitions can be found in \cite{andrews2004integer}.

With $\mathcal{P}_n$ we mean the set of all partitions of the positive integer $n$. For $\lambda  \in\mathcal{P}_n$, we write
$|\lambda|=n$. Denote the parts of the partition by $\lambda = (\lambda_1,\lambda_2,\dotsc,\lambda_N)$ where $\lambda_1 \geq \lambda_2 \geq \dotsb \geq \lambda_N$ and $N=N(\lambda)$ is the number of parts of $\lambda$. Let $r_i=r_i(\lambda)\geq 0$ denote the number of parts of size $i$. Thus, $N=\sum_{i=1}^{n}r_i$ and $n=\sum_{i=1}^N \lambda_i=\sum_{i=1}^{n}ir_i$.

An integer partition $\lambda$ can be represented by a Young diagram drawn as left-aligned rows of squares in the first quadrant such that the $i$th row from the bottom has length $\lambda_i$. The diagrams in the first three levels of Young's lattice can be seen in Figure~\ref{fig:youngslattice}. With $\lambda\in\mathcal{P}_n$ we mean either the partition or its corresponding Young diagram.

\ytableausetup{boxsize=8pt}
\begin{figure}[h]
\begin{tikzpicture}[>=stealth',shorten >=1pt,auto,node distance=3cm,semithick,scale=1.6]
\tikzstyle{every state}=[fill=none,draw=black,text=black,radius=3cm]

\node[inner sep=2pt,minimum size=2pt] (41) at (-2,3){
\begin{ytableau}
 \\
 \\
 \\
 \\
\end{ytableau}
};
\node[inner sep=2pt,minimum size=2pt] (31) at (-1,2) {
\begin{ytableau}
 \\
 \\
 \\
\end{ytableau}
};
\node[inner sep=2pt,minimum size=2pt] (42) at (-1,3) {
\begin{ytableau}
 \\
 \\
 & \\
\end{ytableau}
};
\node[inner sep=2pt,minimum size=2pt] (32) at (0,2) {
\begin{ytableau}
 \\
 \, & \\
\end{ytableau}
};
\node[inner sep=2pt,minimum size=2pt] (43) at (0,3) {
\begin{ytableau}
\, & \\
 & \\
\end{ytableau}
};
\node[inner sep=2pt,minimum size=2pt] (44) at (1,3) {
\begin{ytableau}
 \\
 & & \\
\end{ytableau}
};
\node[inner sep=2pt,minimum size=2pt] (33) at (1,2) {
\begin{ytableau}
\, & & \\
\end{ytableau}
};
\node[inner sep=2pt,minimum size=2pt] (45) at (2,3) {
\begin{ytableau}
\, & & & \\
\end{ytableau}
};
\node[inner sep=2pt,minimum size=2pt] (21) at (-0.5,1.1) {
\begin{ytableau}
 \\
 \\
\end{ytableau}
};
\node[inner sep=2pt,minimum size=2pt] (22) at (0.5,1.1) {
\begin{ytableau}
\, & \\
\end{ytableau}
};
\node[inner sep=2pt,minimum size=2pt] (0) at (0,0.3) {
\begin{ytableau}
 \\
\end{ytableau}
};
\path (0) edge node {} (21);
\path (0) edge node {} (22);
\path (21) edge node {} (31);
\path (21) edge node {} (32);
\path (22) edge node {} (32);
\path (22) edge node {} (33);
\path (31) edge node {} (41);
\path (31) edge node {} (42);
\path (32) edge node {} (42);
\path (32) edge node {} (43);
\path (32) edge node {} (44);
\path (33) edge node {} (44);
\path (33) edge node {} (45);
\end{tikzpicture}
\caption{Young's lattice.}
\label{fig:youngslattice}
\end{figure}

\subsubsection{Description of the process}
To describe the \UD-chain $\textsc{derow-row}(\mu)$, we need a tool introduced
in \cite{Eriksson2012575}, used to associate squares with row-lengths and
row-lengths with corners in a Young diagram.

\begin{definition}
Consider some given Young diagram $\lambda$. For any of its squares $s$ let $\kappa(s)$ denote the length of the row to which $s$ belongs. If $\kappa$ is a row length, let $\omega(\kappa)$ and $\iota(\kappa)$ denote the unique outer corner and inner corner, respectively, for which the row coordinate is $\kappa$:
\begin{align*}
\omega(\kappa) & = \left( \kappa, \max\{i\,|\,\lambda_i=\kappa\} \right), \\
\iota(\kappa) & =
\begin{cases}
\left( \kappa, \max\{i\,|\,\lambda_i>\kappa\} \right) & \text{if }\kappa<\lambda_1 \\
(\lambda_1,0) & \text{if }\kappa=\lambda_1
\end{cases}
.
\end{align*}
\end{definition}

Consider a current Young diagram $\lambda$. The action of the down rule \textsc{derow} is defined by choosing a non-empty row $i$ uniformly at random and removing the corresponding outer corner $\omega(\lambda_i)$.

The action of the up rule $\textsc{row}(\mu)$ is defined as follows: With probability $\mu$ create a new row of length 1. Otherwise make a uniformly random choice of a row $i$ among the $N(\lambda)$ non-empty rows and insert a new square at the corresponding inner corner $\iota(\lambda_i)$.

The process $\textsc{derow-row}(\mu)$ is the up-and-down process on $\mathcal{P}_n$ (for some $n\ge 2$) using these up and down rules.

\subsubsection{The stationary distributions}
The following result is a refinement of equations (12) and (13) in
\cite{Eriksson2012575} for the stationary distributions of the process $\textsc{derow-row}(\mu)$.

\begin{lemma}
\label{lem:derowrow}
For any $0<\mu<1$, the stationary distribution over the partitions in $\mathcal{P}_i\cup\mathcal{P}_{i+1}$ in the process
$\textsc{row-derow}(\mu)$ is
\begin{equation}
\label{eq:derowrowlem:pevenpodd}
{\pi}_i^{\UD}(\lambda)= \frac12(1-\mu)^{|\lambda|-N(\lambda)}\mu^{N(\lambda)-1}\frac{N(\lambda)!}{\prod_k r_k(\lambda)!}.
\end{equation}
\end{lemma}
\begin{proof}
Up to a normalization constant, ${\pi}_i^{\UD}(\lambda)$
is given in (12) and (13) in \cite{Eriksson2012575}, namely,
\begin{equation}
\label{eq:derowrow_even}
{\pi}_{i}^{\UD}(\lambda) =
\frac{ \bigl( \frac{\mu}{1-\mu} \bigr)^{N(\lambda)} N(\lambda)! }{\prod_k r_k(\lambda)!}\cdot
\begin{cases}
\frac12c_i
& \text{if $|\lambda|=i+1$,} \\
\frac12c_i/(1-\mu)
& \text{if $|\lambda|=i$.}
\end{cases}
\end{equation}
for some constant $c_i$ (independent of $\lambda$).
It follows that $c_{i+1}/(1-\mu)=c_i$ for each $i$, and since
$c_1=(1-\mu)^2/\mu$ we must have $c_i=(1-\mu)^{i+1}/\mu$.
\end{proof}

Combining the lemma with Theorem~\ref{th:seqandstatdistr} yields a formula for
the distributions of the $U$-chain induced by the up rule $\textsc{row}(\mu)$.
\begin{corollary}
\label{cor:iff}
The up rule $\textsc{row}(\mu)$ and the down rule $\textsc{derow}$ are compatible and
the distribution of the $\textsc{row}(\mu)$ process on $\mathcal{P}_i$ is given by the probability function
\[
p(\lambda) = (1-\mu)^{i-N(\lambda)}\mu^{N(\lambda)-1}\frac{N(\lambda)!}{\prod_k r_k(\lambda)!}.
\]
\end{corollary}
In fact the derivation of equations (12) and (13) in \cite{Eriksson2012575} reveals that the stationary $\textsc{derow-row}(\mu)$ process is reversible, and hence, by Theorem~\ref{th:chainsandtoggling},
$\textsc{row}(\mu)$ and \textsc{derow} are strongly compatible.

\section{Limiting objects}
\label{sec:limithapes}
The original motivation for the current work has been the study of limiting objects, in particular limit shapes of Young diagrams under stochastic processes as initiated in \cite{Eriksson2012575}. As a corollary to Theorem~\ref{th:seqandstatdistr}, we can now prove a limit shape conjecture in \cite{Eriksson2012575} (Conjecture 2).

First, let us remind ourselves what is meant by a \emph{scaling} of a Young diagram. As the number of squares $n$ grows we need to rescale the diagram to achieve any limiting behaviour. Following \cite{Eriksson2012575} and \cite{VershikStatMech}, a diagram is rescaled using a \emph{scaling factor} $a_n>0$ such that all row lengths are multiplied by $1/a_n$ and all column heights are multiplied by $a_n/n$, yielding a constant diagram area of 1.

\begin{corollary}
\label{cor:iff2}
Let $\mu_n\log(\mu_n n)\rightarrow 0$ and $\mu_n n\rightarrow\infty$ as $n\rightarrow\infty$ and choose the scaling $a_n=1/\mu_n$. Then the stationary distribution of the $\textsc{derow-row}(\mu_n)$ process has the limit shape
\[
y(x)=e^{-x}.
\]
\end{corollary}
\begin{proof}
By Corollary \ref{cor:iff}, the stationary distribution for $\textsc{derow-row}(\mu)$ equals the distribution for $\textsc{row}(\mu)$. By Theorem 4 in \cite{Eriksson2012575}, the limit shape for this distribution with $\mu=\mu_n$ is $y=e^{-x}$ as long as $\mu_n\log(\mu_n n)\rightarrow 0$ and $\mu_n n\rightarrow\infty$ as $n\rightarrow\infty$. Therefore, under these conditions the limit shape for $\textsc{derow-row}(\mu_n)$ must also be $y=e^{-x}$.
\end{proof}

\subsection{The asymptotics of $\mu_n$ in $\textsc{derow-row}(\mu_n)$}
Theorem 4 in \cite{Eriksson2012575} has the conditions $\mu_n n\rightarrow\infty$ and $\mu_n\log(\mu_n n)\rightarrow 0$ as $n\rightarrow\infty$. After scaling a Young diagram with scaling factor $a_n=1/{\mu_n}$, each square will have width $1/a_n=\mu_n$ and height $a_n/n=\frac{1}{\mu_n n}$. In order for the boundary of the Young diagram (which is the object whose limiting behaviour one considers) to resolve properly as $n\rightarrow\infty$, we must at least have
\begin{equation}
\label{eq:muasymptotics}
\mu_n\rightarrow 0 \;\text{ and }\; \mu_n n\rightarrow\infty \;\text{ as }\; n\rightarrow\infty.
\end{equation}

However, Theorem 4 in \cite{Eriksson2012575} uses the stronger assumption $\mu_n\log(\mu_n n)\rightarrow 0$ as $n\rightarrow\infty$. It is still to be investigated whether this can be relaxed in Corollary \ref{cor:iff2}.

\subsection{A generalization of limiting objects}
One may also generalize the concepts of limiting object and limit distribution. In order to be able to talk about a generic process having a limiting object and a limit distribution, let $S$ be a separable metric space, called a \emph{limit space}, with the metric $d_S$. 

For $n=0,1,\dotsc$ let $R_n$ be the countable set of states reachable after $n$ steps, and let $f_n\colon R_n\to S$ be a ``scaling'' function. A stochastic process $(X_n\in R_n)_{n=0}^\infty$ having a limit distribution (with respect to the metric $d_S$) corresponds to the convergence of the random variables $\{f_n(X_n)\}$ \emph{in distribution} to a random variable $X\in S$; the distribution of $X$ is the limit distribution.

Further, if the limit distribution has all probability mass concentrated in a single point in $S$, in other words, if $X$ is constant, then this point is the \emph{limiting object} of the process (in which case we may equivalently talk about convergence \emph{in probability} of $\{f_n(X_n)\}$ to the limiting object).

In this paper, we have made this concept tangible by studying processes on Young's lattice, where the limit space is the set of decreasing functions in the first quadrant with integral 1. In the next section, we will study processes on the $d$-dimensional nonnegative integer lattice $\NN^d$ where the limit space is the set of nonnegative real $d$-tuples adding to 1. We see a further possible application to this generalization in the study of limits of permutation sequences (see for example \cite{hoppen2013limits}). In this case $S$ is the set of Lebesgue measurable functions $[0,1]^2 \rightarrow [0,1]$ with certain properties.

\newpage
\section{Given an up rule, is there a compatible down rule?}
\label{sec:result2}
In this section we will prove the existence of a down rule compatible with a given up rule. The proof is constructive and we will use the construction in some examples.

\begin{theorem}\label{th:existencedownrule}
Let $U$ be an up rule on an $I$-graded poset $\Omega$, and let
$(X_i)_{i\in I}$ be a $U$-chain. Then
$(X_i)$ is a $D$-chain for some down rule $D$ on $\Omega$ if and
only if $\Prob(X_i=u)=0$ implies $\Prob(X_{i+1}=v)=0$ for any
$i,i+1\in I$ and any $\Omega_i\ni u\lessdot v\in\Omega_{i+1}$.
In particular this holds when $(X_i)$ is positive.
\end{theorem}
\begin{proof}
By Theorem~\ref{th:chainsandtoggling}, $(X_i)$ is a $D$-chain
if and only if $X_i,X_{i+1},X_i,X_{i+1},\dotsc$ is a
\UD-chain for any $i\in I$. Since $X_i,X_{i+1},X_i,X_{i+1},\dotsc$
is a reversible Markov chain (see any text book on Markov chains, for
instance \cite{ross2014introduction}), it is a \UD-chain if and only
if $\Prob(X_i=u)U(u\rightarrow v)=\Prob(X_{i+1}=v)D(v\rightarrow u)$
for any $u,v\in\Omega$. A $D$ satisfying that equation
can be chosen by letting
\begin{equation}
\label{eq:downrule}
D(v\rightarrow u) = \frac{\Prob(X_i=u)U(u\rightarrow v)}{\Prob(X_{i+1}=v)}
\end{equation}
but this is possible only unless $\Prob(X_i=u)=0$ and $\Prob(X_{i+1}=v)>0$ for some $u\lessdot v$
or $\Prob(X_i=u)>0$ and $\Prob(X_{i+1}=v)=0$
for some $u\lessdot v$. The latter is impossible since
$\Prob(X_{i+1}=v)\ge\Prob(X_i=u)U(u\rightarrow v)$.
\end{proof}

\subsection{Processes on the $d$-dimensional nonnegative integer lattice}
\label{sec:lattice}
In this section we will demonstrate two applications of Theorem~\ref{th:existencedownrule} on the $d$-dimensional nonnegative
integer lattice $\NN^d$. We will construct down rules
compatible with given up rules. We shall also see examples
of processes both with and without a limiting object.

\subsubsection{The lattice $\NN^d$}
For a positive integer $d$, let $(\NN^d,\le)$ be the poset of nonnegative integer $d$-tuples 
$(x_1,\dotsc,x_d)$ ordered component-wise, i.e. for
$x=(x_1,\dotsc,x_d), y=(y_1\dotsc,y_d)\in \NN^d$,
we have $x\le y$ if $x_i \le y_i$ for all $i=1,\dotsc,d$.
For $x,y\in \NN^d$, we have
\begin{align*}
x \lor y &= (\max(x_1,y_1),\dotsc,\max(x_d,y_d)) \text{ and} \\
x \land y &= (\min(x_1,y_1),\dotsc,\min(x_d,y_d)),
\end{align*}
so $\NN^d$ is a lattice, which we will refer to as \emph{the $d$-dimensional nonnegative integer lattice}.

Obviously, $\NN^d$ is
graded with the rank function
\[
|x| =\sum_{j=1}^d x_j
\]
and for $n\ge 0$, the level set $\NN_n^d=\{ x\in\NN^d : |x|=n \}$ is the set of weak compositions of $n$.

\subsubsection{A limiting object on $\NN^d$}
As an analogy to the concepts of limit shapes for birth- and birth-and-death processes on Young diagrams we will here define a limiting object for processes on $\NN^d$.

If we divide the coordinates of a lattice point reached in an up process on $\NN^d$ by the number of steps taken in the process, the result is a point in the $(d-1)$-dimensional simplex
$\Delta^{d-1}=\{ (x_1,\dotsc,x_d)\in[0,1]^d \;|\; x_1+\dotsb+x_d=1 \}$.
Under this scaling we define a limiting object as follows.

\begin{definition}
	For an up process on $\NN^d$, let $X_n=(X^1_n,\dotsc,X^d_n)$ be the lattice point after $n$ steps. 
	A point $x\in\Delta^{d-1}$ is a limit point
	of the up process if for any $\eps>0$ we have
	\[
	\lim_{n\rightarrow\infty}\Prob\left( \left| \tfrac{1}{n}X_n-x \right| < \eps \right) = 1.
	\]
\end{definition}

\subsubsection{Notation}
For $x=(x_1,\dotsc,x_d)\in \NN^d$, let
\begin{align*}
x_{(i)} &= (x_1,\dotsc, x_i-1, \dotsc, x_d) \text{ and} \\
x^{(i)} &= (x_1,\dotsc, x_i+1, \dotsc, x_d)
\end{align*}
for $i=1,\dotsc,d$. An up process on $\NN^d$ starts at $(0,\dotsc,0)$. In each step, the up rule $U$ increases the rank of the current state by~1, i.e.\ increments exactly one component. Thus, $U$ is determined by the probabilities $U(x\rightarrow x^{(i)})$ for $i=1,\dotsc,d$.

For $n\in\NN$, let $p_n\colon\NN^d_n \to [0,1]$ be the probability function on $\NN^d_n$ induced by the up rule $U$.

\subsubsection{An up process on $\NN^d$ with a limiting object}
\label{sec:ex:lattice1}
Consider the up rule $U$ on $\NN^d$ governed by the constant transition probabilities $U(x\rightarrow x^{(i)})=\nu_i$. We may use Theorem~\ref{th:existencedownrule} to conclude that there is a down rule $D$ compatible with $U$. Let us construct it!

First of all, the probability to reach $x=(x_1,\dotsc,x_d)$ (after $n=|x|$ steps in the up process induced by $U$) is clearly
\[
p_{n}(x_1,\dotsc,x_d) = \frac{n!}{x_1!\dotsb x_d!}{\nu_1}^{x_1}\dotsb {\nu_d}^{x_d}.
\]
As a consequence,
\[
p_{n-1}(x_{(i)}) = \frac{(n-1)!}{x_1!\dotsb (x_i-1)! \dotsb x_d!}{\nu_1}^{x_1}\dotsb {\nu_i}^{x_i-1} \dotsb {\nu_d}^{x_d}.
\]
We use \eqref{eq:downrule} to compute the transition probability
$D(x\rightarrow x_{(i)})$ from $x\in \NN^d_n$ to $x_{(i)}\in \NN^d_{n-1}$:
\[
D(x\rightarrow x_{(i)})
= \frac{p_{n-1}(x_{(i)}) U(x_{(i)}\rightarrow x)}{p_n(x)}
= \frac{p_{n-1}(x_{(i)}) \nu_i}{p_n(x)}
= \frac{x_i}{n} = \frac{x_i}{|x|}.
\]
Here we see that an up rule employing the degree of freedom parameters $\nu_1,\dotsc,\nu_d$ has a compatible down rule with no such degree of freedom present. As we saw in Corollary~\ref{cor:iff}, this is also the case with the up rule in the process $\textsc{row}(\mu)$ having a compatible down rule \textsc{derow}, void of the degree of freedom parameter $\mu$. 

\begin{proposition}
The up process on $\NN^d$ using the up rule $U(x\to x^{(i)})=\nu_i$ has the limit point $(\nu_1,\dotsc,\nu_d)$.
\end{proposition}
\begin{proof}
For this process we have $\Prob(X^i_n=k) = \binom{n}{k}{\nu}^k(1-\nu)^{n-k}$ for $i=1,\dotsc,d$, i.e.\ $X^i_n \sim \text{Bin}(n,\nu_i)$, which means $E(X^i_n)/n=\nu_i$.
Thus, the limit point is the $d$-tuple $(\nu_1,\dotsc,\nu_d)$.
\end{proof}

\subsubsection{An up process on $\NN^d$ without a limiting object}
\label{sec:ex:lattice2}
We will now consider an up process on $\NN^d$ starting at $(0,\dotsc,0)$ induced by an up rule that depends on the current level.

\begin{proposition}
\label{prop:ex2}
The process on $\NN^d$ induced by the up rule $U$ governed by the transition probabilities
\begin{equation}
\label{eq:ex2}
U(x\rightarrow x^{(i)})=\frac{x_i+1}{x_1+\dotsc+x_d+d}=\frac{x_i+1}{|x|+d}
\end{equation}
has a uniform distribution on $\NN^d_n$, for all $n\ge 0$.
\end{proposition}
\begin{proof}
We prove this by induction over the number of steps $n$ in the process. 
First of all we observe that $|\NN^d_n|=\binom{d+n-1}{d-1}$.

For $n=0$, we have a trivial uniform distribution.

For $n\ge 0$, assume that the distribution on $\NN^d_n$ under this process is uniform, i.e., $p_n(x)=|\NN^d_n|^{-1}={\binom{d+n-1}{d-1}}^{-1}$ for all $x\in \NN^d_n$. We want to prove that the distribution over $\NN^d_{n+1}$ is also uniform, i.e.,
$p_{n+1}(x)=|\NN^d_{n+1}|^{-1}={\binom{d+n}{d-1}}^{-1}$.

Let $x=(x_1,\dotsc,x_d)\in \NN^d_{n+1}$ be a fixed element. Of course $|x|=n+1$. Now, $x$ can be reached from any of the elements $x_{(1)},\dotsc,x_{(d)}\in \NN^d_n$. By \eqref{eq:ex2},
\[
U(x_{(j)} \rightarrow x)=\frac{(x_j-1)+1}{|x_{(j)}|+d}=\frac{x_j}{n+d},
\]
and by the induction hypothesis, $p_n(x_{(j)})={\binom{d+n-1}{d-1}}^{-1}$,
so the probability $p_{n+1}(x)$ for $x\in \NN^d_{n+1}$ is
\begin{eqnarray*}
p_{n+1}(x) & = & \sum_{j=1}^d p_n(x_{(j)}) P(x_{(j)}\rightarrow x) \\
& = & {\binom{d+n-1}{d-1}}^{-1}\sum_{j=1}^d \frac{x_j}{n+d} \\
& = & \frac{(d-1)!n!(n+1)}{(d+n-1)!(n+d)} = \frac{(d-1)!(n+1)!}{(n+d)!} 
 =  {\binom{d+n}{d-1}}^{-1},
\end{eqnarray*}
so the distribution is uniform also on $\NN^d_{n+1}$, and the result follows by induction.
\end{proof}
Since the distribution is uniform, this process cannot have a limiting object.

As in Section \ref{sec:ex:lattice1}, let us use \eqref{eq:downrule} to construct the down rule compatible with the up rule in Theorem~\ref{prop:ex2}. We get the probability $D(x\to x_{(i)})$ for moving from $x\in \NN^d_n$ to $x_{(i)}\in \NN^d_{n-1}$ by
\begin{eqnarray*}
D(x\rightarrow x_{(i)}) & = & \frac{p_{n-1}(x_{(i)}) P(x_{(i)}\rightarrow x)}{p_n(x)} \\
& = &
\frac{{\binom{d+n-2}{d-1}}^{-1} \frac{x_i}{n-1+d}}
{{\binom{d+n-1}{d-1}}^{-1}} 
 =  \frac{\binom{d+n-1}{d-1}}{\binom{d+n-2}{d-1}} \frac{x_i}{n-1+d}
=\frac{x_i}{n}
=\frac{x_i}{|x|}.
\end{eqnarray*}

\section{Given an up or down rule $T$, is there a maximal $T$-chain?}
\label{sec:result3}
Let $T$ be an up or down rule on an $I$-graded poset $\Omega$.
Recall that there is a one-to-one correspondence between $T$-chains
and $T$-sequences, so questions about the existence of maximal
$T$-chains are
equivalent to questions about the existence of maximal $T$-sequences.

If $I$ has a minimal element $m$ there obviously exists a maximal
$U$-sequence for a given up rule $U$:
Just choose any probability distribution
$\pi_m\in M(\Omega_m)$ on level $m$,
and the up rule $U$ will induce a $U$-sequence
$(U^{i-m}\pi_m)_{i\in I}$. However,
if $I$ has no lower bound it is not obvious whether there exists a
maximal $U$-sequence, and, by symmetry, if $I$ has no upper bound
it is not obvious whether there exists a maximal $D$-sequence for a given
down rule $D$. Here is an example where no such $D$-sequence exists.

\begin{example}
Let $\Omega$ be the two-dimensional integer lattice $\ZZ^2$
with the partial order
$(x,y)\le (x',y')$ if $x\le x'$ and $y\le y'$. It is $\ZZ$-graded
by $\rho(x,y)=x+y$. Let $D$ be the down rule with probability $1/2$
for each edge in the Hasse diagram.
If we start at any element of high rank $n$
and follow down edges randomly according to the down rule, the chance of
hitting any particular element of rank zero is very small;
it tends to zero as $n$ grows. Thus, in a maximal $D$-sequence
any rank-zero element must be given the probability zero, which is
impossible.
\end{example}

The phenomenon in the above example cannot happen
if the level sets $\Omega_i$ are finite, and,
as we will see in Theorem~\ref{th:finiteexistence}, in this case
there is always a maximal $T$-sequence.
As the following example
reveals, however, the existence of a \emph{positive}
maximal $T$-sequence is not guaranteed.

\begin{example}
\label{ex:n2}
Look at the two-dimensional
nonnegative integer lattice $\NN^2$ with the down rule given by the probability
$1/2$ at every edge in the Hasse diagram,
except for the leftmost and rightmost edges which
must have probability one; see Figure~\ref{fig:n2}.
If we start at any element of high rank $n$
and follow down edges randomly according to the down rule, with very high
probability we will walk into the left or right border before we reach
level 2. Thus, the probability of reaching the middle element $(1,1)$ at
level 2 tends to zero as $n$ grows. This means that every $D$-chain
$(X_i)$ must have $\Prob(X_2=(1,1))=0$.
\end{example}

\begin{figure}[h]
	\begin{tikzpicture} 
	\draw (0,0) -- (-3,3);
	\draw (0,0) -- (3,3);
	\draw (-1,1) -- (0,2);
	\draw (1,1) -- (0,2);
	
	\draw (-2,2) -- (-1,3);
	\draw (-1,3) -- (0,2);
	\draw (0,2) -- (1,3);
	\draw (1,3) -- (2,2);

	\draw[->] (-1,1) -- (-0.5,0.5);
	\draw[->] (1,1) -- (0.5,0.5);
	\draw[->] (-2,2) -- (-1.5,1.5);
	\draw[->] (0,2) -- (-0.5,1.5);
	\draw[->] (0,2) -- (0.5,1.5);
	\draw[->] (2,2) -- (1.5,1.5);

	\draw[->] (-3,3) -- (-2.5,2.5);
	\draw[->] (-1,3) -- (-1.5,2.5);
	\draw[->] (-1,3) -- (-0.5,2.5);	
	\draw[->] (1,3) -- (0.5,2.5);
	\draw[->] (1,3) -- (1.5,2.5);
	\draw[->] (3,3) -- (2.5,2.5);
	
	\draw[fill] (0,0) circle (.06);
	\draw[fill] (-1,1) circle (.06);
	\draw[fill] (-2,2) circle (.06);
	\draw[fill] (0,2) circle (.06);
	\draw[fill] (1,1) circle (.06);
	\draw[fill] (2,2) circle (.06);
	
	\draw[fill] (-3,3) circle (.06);
	\draw[fill] (-1,3) circle (.06);
	\draw[fill] (1,3) circle (.06);
	\draw[fill] (3,3) circle (.06);
	
	\node[left] at (-0.5,0.5) {$1$};
	\node[right] at (0.5,0.5) {$1$};
	\node[left] at (-1.5,1.45) {$1$};
	\node[left] at (-0.5,1.55) {$\frac{1}{2}$};
	\node[right] at (0.5,1.55) {$\frac{1}{2}$};
	\node[right] at (1.5,1.45) {$1$};

	\node[left] at (-2.5,2.5) {$1$};
	\node[left] at (-1.5,2.5) {$\frac{1}{2}$};
	\node[right] at (-0.5,2.5) {$\frac{1}{2}$};
	\node[left] at (0.5,2.5) {$\frac{1}{2}$};
	\node[right] at (1.5,2.5) {$\frac{1}{2}$};
	\node[right] at (2.5,2.5) {$1$};
	
	\draw[dotted] (-3,3) -- (-3.3,3.3);
	\draw[dotted] (-3,3) -- (-2.7,3.3);
	
	\draw[dotted] (-1,3) -- (-1.3,3.3);
	\draw[dotted] (-1,3) -- (-0.7,3.3);
	
	\draw[dotted] (1,3) -- (1.3,3.3);
	\draw[dotted] (1,3) -- (0.7,3.3);
	
	\draw[dotted] (3,3) -- (3.3,3.3);
	\draw[dotted] (3,3) -- (2.7,3.3);	
	
	\end{tikzpicture}
	\caption{The down rule on $\NN^2$ in example~\ref{ex:n2}.}
	\label{fig:n2}
\end{figure}
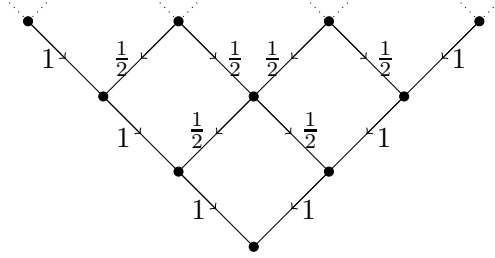

We will present three theorems about the existence of $T$-sequences
and positive $T$-sequences. Their proofs all depend on the following
lemma.

\begin{lemma}\label{lm:compact}
Let $T$ be an up or down rule on an $I$-graded poset $\Omega$.
Let $(C_i\subseteq M(\Omega_i))_{i\in I}$ be a sequence of compact
sets and suppose for any $m\le n$
in $I$ there exists a $T$-sequence $(\pi'_i\in C_i)_{i\in[m,n]}$.
Then there exists a maximal $T$-sequence $(\pi_i\in C_i)_{i\in I}$.
\end{lemma}
\begin{proof}
Without loss of generality, we may assume that $T=U$ is an up rule.

Let $\Pi$ be the product space $\Pi=\prod_{i\in I}C_i$.
For any non-maximal $k\in I$, let
\[
\Pi_k:=\{(\pi_i)\in
\Pi\ \vert\ U\pi_k=\pi_{k+1}\}.
\]
Let us first show that the set $\Pi_k$ is closed.

Clearly, $\Pi_k$ is homeomorphic to
the product of $\prod_{i\in I\setminus\{k,k+1\}}C_i$ and
the graph $G\subseteq C_k\times C_{k+1}$ of the restriction of $U$ to
$C_k\cap U^{-1}C_{k+1}$. The map $U$ has operator norm 1,
so it is continuous.
Hence, the preimage $U^{-1}C_{k+1}$ is closed and so is
$C_k\cap U^{-1}C_{k+1}$. It follows that the graph $G$ is closed
and hence $\Pi_k$ is closed, being homeomorphic to a product of
closed sets.

By Tychonoff's theorem $\Pi$ is compact,
and therefore the closed subsets $\Pi_k$ are compact, and
so is the intersection $\Lambda$ of all $\Pi_k$.
Clearly, $\Lambda$ is precisely the set of maximal $U$-sequences,
so our task is to show that
$\Lambda$ is nonempty. To that end, suppose it is
empty and define $U_k=\Pi\setminus \Pi_k$.
Then $\bigcup U_k=\Pi\setminus\bigcap \Pi_k=\Pi$
so the family of all sets
$U_k$ is an open cover of $\Pi$. Since $\Pi$ is compact there
is an open
subcover $\{U_k\}_{k\in F}$, where $F$ is a finite set of
non-maximal elements in $I$.
Choose $m$ and $n$ in $I$ so that $m\le k<n$ for any $k\in F$.
Now, by the assumption in the
theorem there is a
$U$-sequence $(\pi'_i\in C_i)_{i\in[m,n]}$
For $i\in I\setminus[m,n]$, let $\pi'_i$ be an arbitrary element
in $C_i$.
The sequence $(\pi'_i)_{i\in I}$ so obtained is a point outside the union
$\bigcup_{k\in F}U_k$, which
contradicts the fact that
$\{U_k\}_{k\in F}$ covers $\Pi$. Hence,
our supposition that $\Lambda$ is empty is false.
\end{proof}

Our first existence theorem states that if there
exists a $T$-sequence for any finite subinterval of $I$, and these
sequences are uniformly bounded in a certain sense, then there
exists a maximal $T$-sequence.

For $f,g\in\ell_1(\Omega)$ we will write $f\le g$ if
$f(u)\le g(u)$ for all $u\in \Omega$.
\begin{theorem}\label{th:generalexistence}
Let $T$ be an up or down rule on an $I$-graded poset $\Omega$.

Suppose there is a sequence $(\hat{b}_i\in\ell_1(\Omega_i))_{i\in I}$
such that for any $m\le n$
there exists a $T$-sequence $(\pi_i)_{i\in [m,n]}$
with $\pi_i\le\hat{b}_i$ for any $i\in [m,n]$. Then there exists
a maximal $T$-sequence.
\end{theorem}
\begin{proof}
For any $i\in I$, let
$C_i=\{\pi\in M(\Omega_i)\ :\ \pi\le\hat{b}_i\}$.
If we can show that $C_i$ is compact, the theorem will follow
from Lemma~\ref{lm:compact}. We can write
$C_i=M(\Omega_i)\cap L_i$ where
\[
L_i=\{\pi\colon\Omega_i\to\RR\ | \;
0\le\pi\le\hat{b}_i\}\subset\ell_1(\Omega_i).
\]
Since $M(\Omega_i)$ is closed it suffices to show that
$L_i$ is compact. By the dominated convergence theorem
$L_i$ is homeomorphic to the product space
$\prod_{u\in\Omega_i}[0,\hat{b}_i(u)]$ which is compact by
Tychonoff's theorem.
\end{proof}

For the existence of a \emph{positive} maximal
$T$-sequence, it is not enough to assume the existence of finite $T$-sequences that are uniformly bounded
from above; they must be uniformly bounded from above and from below
simultaneously!

Let $\ell_1^+(\Omega_i)=\{f\in\ell_1(\Omega_i)\,:\,f(u)>0\
\text{for any}\ u\in\Omega_i\}$ denote the set of strictly positive
$\ell_1$-functions on $\Omega_i$.
\begin{theorem}\label{th:generalpositiveexistence}
Let $T$ be an up or down rule on an $I$-graded poset $\Omega$.

Suppose there are sequences
$(\check{b}_i\in\ell_1^+(\Omega_i))_{i\in I}$
and
$(\hat{b}_i\in\ell_1^+(\Omega_i))_{i\in I}$
such that for any $m\le n$ in $I$
there exists a $T$-sequence $(\pi_i)_{i\in[m,n]}$ with
$\check{b}_i\le\pi_i\le\hat{b}_i$ for any $i\in[m,n]$.
Then there exists a positive maximal $T$-sequence.
\end{theorem}
\begin{proof}
The proof is completely analogous to that of
Theorem~\ref{th:generalexistence}, but with
$L_i=\{\pi\colon\Omega_i\to\RR\ :\
\check{b}_i\le\pi\le\hat{b}_i\}\subset\ell_1(\Omega_i)$.
\end{proof}

In most combinatorial applications, the level sets $\Omega_i$ are
finite. In that case, the uniform upper bound $(\hat{b}_i)$
in the assumption in
Theorem~\ref{th:generalexistence} automatically exists. Furthermore, the
requirement of uniformicity of the lower bound $(\check{b}_i)$ in Theorem~\ref{th:generalpositiveexistence} can be relaxed.

The following theorem is stated for an up rule, but the dual statement
for a down rule is of course equivalent.

\begin{theorem}\label{th:finiteexistence}
Let $U$ be an up rule on an $I$-graded poset $\Omega$ with
finite level sets $\Omega_i$.
Then there exists a maximal $U$-sequence,
and there exists a positive
maximal $U$-sequence
if and only if there is a sequence
$(b_i\colon\Omega_i\to(0,\infty))_{i\in I}$
with the property that for any $m\le n$
in $I$ there is a distribution $\pi_{m,n}\in M(\Omega_m)$ such that
$U^{n-m}\pi_{m,n}\ge b_n$.
\end{theorem}
\begin{proof}
Define a sequence $(\hat{b}_i\in\ell_1(\Omega_i))_{i\in I}$
by letting $\hat{b}_i(u)=1$ for any $u\in\Omega_i$. Now
Theorem~\ref{th:generalexistence} guarantees the existence of
a maximal $U$-sequence.

Next, suppose there is a sequence
$(b_i\colon\Omega_i\to(0,\infty))_{i\in I}$
and distributions $(\pi_{m,n}\in M(\Omega_m))_{m\le n}$
with the property given in the theorem.
Let $(\gamma_i)_{i\in I}$ be positive numbers adding to one,
and for each $i\in I$ put $\check{b}_i=\gamma_i b_i$.

Now, fix $m\le n$ in $I$.
Let $\Gamma=\sum_{i\in [m,n]}\gamma_i$
and define
\[
\pi=\frac1{\Gamma}\sum_{i\in[m,n]} \gamma_i\pi_{m,i}.
\]
Since $U$ is a linear operator, for any $j\in[m,n]$ we have
\[
U^{j-m}\pi=
\frac1{\Gamma}\sum_{i\in[m,n]} \gamma_i U^{j-m}\pi_{m,i}
\ge
\frac{\gamma_j}{\Gamma}U^{j-m}\pi_{m,j}
\ge
\frac{\gamma_j}{\Gamma}b_j
=\frac{\check{b}_j}{\Gamma}
\ge \check{b}_j
\]
and hence $(U^{j-m}\pi)_{j\in[m,n]}$ is a $T$-sequence such that
$U^{j-m}\pi\ge\check{b}_j$ for any $j\in[m,n]$.
Theorem~\ref{th:generalpositiveexistence} now yields the existence
of a positive maximal $U$-sequence.
\end{proof}

Theorem~\ref{th:finiteexistence} guarantees the existence of a
maximal $T$-sequence if all level sets are finite, but
in general it is not possible to extend
a $T$-sequence for a subinterval $J\subset I$ to a maximal $T$-sequence.
For example, consider the poset in
Figure~\ref{fig:notconstructive} with a defined down rule $D$.
The distributions $\pi_0\in M(\Omega_0)$ and $\pi_1\in M(\Omega_1)$ given by
$\pi_0(\hat{0})=1$, $\pi_1(s_1)=1/4$ and $\pi_2(s_2)=3/4$ constitute
a $D$-chain, but this $D$-chain can obviously not be extended to the
top level.
\begin{figure}[h]
\setlength{\unitlength}{0.06cm}
\begin{picture}(80,80)

\put(40,10){\circle*{3}}
\put(10,40){\circle*{3}}
\put(70,40){\circle*{3}}
\put(40,70){\circle*{3}}

\put(10,40){\vector(1,-1){16}}
\put(70,40){\vector(-1,-1){16}}
\put(40,70){\vector(-1,-1){16}}
\put(40,70){\vector(1,-1){16}}

\put(10,40){\line(1,-1){30}}
\put(70,40){\line(-1,-1){30}}
\put(40,70){\line(-1,-1){30}}
\put(40,70){\line(1,-1){30}}

\put(38.6,3.5){$\hat{0}$}
\put(20,20){$1$}
\put(57,20){$1$}
\put(18,55){$\frac{1}{2}$}
\put(59,55){$\frac{1}{2}$}
\put(38.6,73){$\hat{1}$}

\put(3,39){$s_1$}
\put(72,39){$s_2$}

\end{picture}
\caption{A graded poset with a given down rule $D$.}
\label{fig:notconstructive}
\end{figure}
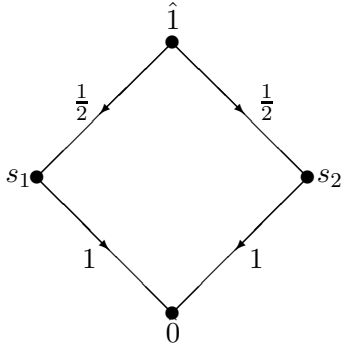

\section{General homogeneous Markov chains}
Finally, we shall use Theorem~\ref{th:generalexistence} to prove a more general result about the 
existence of Markov chains with time interval $\ZZ$. As far as we know, this has not been treated
before.

Let $S$ be a countable (or finite) state space and let $P$ be a transition
function on $S$, that is, an assignment of a probability $P(s\to s')$ to
each pair $(s,s')\in S\times S$ such that $\sum_{s'\in S}P(s\to s')=1$ for any $s\in S$. For any initial (random) state $X_0$, there is a unique Markov chain $X_0,X_1,\dotsc$ with transition function $P$, but is there a Markov chain 
$\dotsc,X_{-1},X_0,X_1,\dotsc$ with time interval $\ZZ$? The following ``homogeneous'' version of Theorem~\ref{th:generalexistence} answers that
question.

\begin{theorem}
Let $S$ be a state space and let $P$ be a transition function on $S$.
Suppose there exists a sequence of functions $(\hat{b}_i \in \ell_1(S))_{i\in\ZZ}$ such that for any
integers $m \le n$ there is a Markov chain $X_m,\dotsc,X_n$ on $S$ with transition function $P$ such
that $\Prob(X_i=s) \le \hat{b}_i(s)$
for any $i\in[m,n]$ and any $s\in S$.
Then there exists a Markov chain $\dotsc,X_{-1},X_0,X_1,\dotsc$ on $S$
with transition function $P$.
\end{theorem}
\begin{proof}
Define a $\ZZ$-graded poset $\Omega=S\times\ZZ$ with $(s,i)\lessdot(s',i+1)$ if $P(s,s')>0$
and define an up rule $U$ on $\Omega$ by letting
\[
U( (s,i)\to(s',i+1) ) = P(s\to s')
\]
for any $i\in\ZZ$ and any $s,s'\in S$. Now the theorem follows from
Theorem~\ref{th:generalexistence}.
\end{proof}

\bibliographystyle{amsplain}
\bibliography{C:/Users/mjn15/Dropbox/C/Documents/LaTeX/Bib/limitshape}
\end{document}